\documentclass[a4paper,11pt]{amsart}
\usepackage{hyperref,fancyhdr,mathrsfs,amsmath,amscd,amsthm,amsfonts,latexsym,amssymb,stmaryrd}
\usepackage[all]{xy}

\DeclareMathOperator{\rank}{rank}

\DeclareMathOperator{\dif}{d}

\newcommand{\Cal}{\mathcal{C}}
\renewcommand{\H}{\mathscr{H}}
\newcommand{\V}{\mathscr{V}}

\newcommand{\Fa}{\mathcal{F}}

\newcommand{\ol}{\mathcal{O}}

\def \a{\alpha}

\def \b{\beta}
\def \e{\eta}

\def \G{\Gamma}

\def \O{\Omega}
\def \phi{\varphi}
\def \Phi{\varPhi}
\def \p{\pi}
\def \r{\rho}
\def \s{\sigma}

\def \R{\mathbb{R}}
\def \Hq{\mathbb{H}\,}
\def \C{\mathbb{C}\,}

\def\widecheckg{g^{\hspace*{-2.5pt}\vbox to 5pt{\hbox to
0pt{\LARGE$\check{}$}}}\hspace*{2pt}}

\def\widecheckl{\lambda^{\hspace*{-3.5pt}\vbox to 8pt{\hbox to
0pt{\LARGE$\check{}$}}}\hspace*{2pt}}

\begin{document}

\title{On the integrability of the\\ 
co-CR quaternionic structures} 
\author{Radu Pantilie}
\email{\href{mailto:radu.pantilie@imar.ro}{radu.pantilie@imar.ro}}
\address{R.~Pantilie, Institutul de Matematic\u a ``Simion~Stoilow'' al Academiei Rom\^ane,
C.P. 1-764, 014700, Bucure\c sti, Rom\^ania}
\subjclass[2010]{Primary 53C28, Secondary 53C26}

\newtheorem{thm}{Theorem}[section]
\newtheorem{lem}[thm]{Lemma}
\newtheorem{cor}[thm]{Corollary}
\newtheorem{prop}[thm]{Proposition}

\theoremstyle{definition}

\newtheorem{defn}[thm]{Definition}
\newtheorem{rem}[thm]{Remark}
\newtheorem{exm}[thm]{Example}

\numberwithin{equation}{section}
 
\begin{abstract} 
We characterise the integrability of any co-CR quaternionic structure in terms of the curvature 
and a generalized torsion of the connection. Also, we apply this result to obtain, for example, the following:\\ 
\indent
$\bullet$ New co-CR quaternionic structures built on vector bundles over a quaternionic 
manifold $M$, whose twistor spaces are holomorphic vector bundles over the twistor space $Z$ of $M$. 
Moreover, all the holomorphic vector bundles over $Z$, 
which are positive and isotypic when restricted to the twistor lines, are obtained this way.\\ 
\indent 
$\bullet$ Under generic dimensional conditions, any manifold endowed with an almost \mbox{$f$-quaternionic} structure 
and a compatible torsion free connection is, locally, a product of a hypercomplex manifold 
with $({\rm Im}\Hq\!)^k$, for some $k\in\mathbb{N}$\,.
\end{abstract} 

\maketitle 
\thispagestyle{empty} 
\vspace{-8mm}

\section*{Introduction} 

\indent 
It is a basic fact that any notion has an adequate level of generality. For example, anyone with some interest in 
three-dimensional Einstein--Weyl spaces and quaternionic manifolds, should have felt a need for a more general 
geometric notion, with a twistor space endowed with a locally complete family of spheres with positive normal bundle. 
In \cite{fq_2} and \cite{Pan-twistor_(co-)cr_q} (see, also, \cite{fq}\,) it is shown that the \emph{co-CR quaternionic manifolds} 
fulfill the needs for such a notion. Up to now, we know the following manifolds which are endowed with natural 
co-CR quaternionic structures:\\ 
\indent 
\quad(a) the three-dimensional Einstein--Weyl spaces.\\ 
\indent 
\quad(b) the quaternionic manifolds (in particular, the anti-self-dual manifolds).\\ 
\indent 
\quad(c) the local orbit spaces of any nowhere zero quaternionic vector field on a quaternionic manifold.\\ 
\indent 
\quad(d) a principal bundle built over any quaternionic manifold (in particular, $S^{4n+3}$); 
the corresponding twistor space is the product of the sphere with the twistor space of the quaternionic manifold 
($\C\!P^1\times\C\!P^{2n+1}$ for $S^{4n+3}$).\\ 
\indent 
\quad(e) vector bundles over any quaternionic manifold; the twistor spaces are 
holomorphic vector bundles over the twistor space of the quaternionic manifold.\\ 
\indent 
\quad(f) the Grassmannian of oriented three-dimensional vector subspaces of the Euclidean space of dimension $n+1$\,; 
the twistor space is the nondegenerate hyperquadric in the $n$-dimensional complex projective space, $n\geq3$\,.\\ 
\indent 
\quad(g) the complex manifold $M_V$ formed of the isotropic two-dimensional vector subspaces of any complex symplectic vector space $V$; the twistor space is $M_V$ itself.\\ 
\indent 
\quad(h) the space of holomorphic sections of $P\bigl(\ol\oplus\ol(n)\bigr)$ induced by the holomorphic sections of $\ol(m)\oplus\ol(m+n)$ which 
intertwine the antipodal map and the conjugation; the twistor space is $P\bigl(\ol\oplus\ol(n)\bigr)$\,, where $m,n\in\mathbb{N}$ are even, $n\neq0$\,.\\ 
\indent 
\quad(i) the space of holomorphic maps of a fixed odd degree from $\C\!P^1$ to $\C\!P^1$ which commute with the antipodal map; 
the twistor space is $\C\!P^1\times\C\!P^1$.\\ 
(The details for (d), (f), (g) can be found in \cite{fq_2}\,, for (c), (h), (i) in \cite{Pan-twistor_(co-)cr_q}\,, 
whilst the details for (e) will be given in Section \ref{section:co-cr_q_integrab}\,, below.)\\ 
\indent 
In this paper, we settle the problem of finding a useful characterisation for the integrability of the co-CR quaternionic structures. 
For this, we use a seemingly new generalized torsion associated to any connection on a vector bundle $E$, over a manifold $M$, 
endowed with a morphism $\r:E\to TM$ (if $\r$ is an isomorphism this reduces to the classical torsion 
of a connection on a manifold). This is studied in Section \ref{section:gen_torsion}\,, where we show that it provides a necessary 
tool to handle the integrability of distributions defined on Grassmannian bundles.\\ 
\indent 
In Section \ref{section:co-cr_q_integrab}\,, we give the main integrability result (Theorem \ref{thm:co-cr_q_integrab}\,), 
and its first applications. For example, there we prove (Theorem \ref{thm:holo_bundles_over_Z}\,) that the following 
holds for any holomorphic vector bundle $\mathcal{Z}$\,, over the twistor space $Z$ of a quaternionic manifold $M$, 
endowed with a conjugation covering the conjugation of $Z$\,: \emph{if the Birkhoff--Grothendieck decomposition 
of $\mathcal{Z}$ restricted to each twistor line contains only terms of Chern number $m\geq1$ then $\mathcal{Z}$ is the twistor space 
of a co-CR quaternionic manifold, built on the total space of a vector bundle over $M$}. This gives example (e)\,, above, 
and, in particular, for $m=1$ it reduces to \cite[Theorem 7.2]{Sal-dg_qm}\,.\\ 
\indent 
An important particular type of co-CR quaternionic manifolds is provided by the $f$-quaternionic manifolds. 
For example, all of the (a), (b), (d), (f), and (g), above, give such manifolds \cite{fq_2}\,.  
Also, the same applies to example (e) (Theorem \ref{thm:holo_bundles_over_Z}\,) if $m=2$\,.\\ 
\indent  
In Section \ref{section:f-q_integrab}\,, we apply Theorem \ref{thm:co-cr_q_integrab}\,, to study the integrability of 
$f$-quaternionic structures. This leads to Theorem \ref{thm:f-q_generic_dims} by which, under generic dimensional conditions, 
any manifold endowed with an almost $f$-quaternionic structure 
and a compatible torsion free connection is, locally, a product of a hypercomplex manifold 
with $({\rm Im}\Hq\!)^k$, for some $k\in\mathbb{N}$\,.

\section{A generalized torsion} \label{section:gen_torsion} 

\indent 
We work in the smooth and the complex-analytic categories (in the latter case, by the tangent bundle we mean 
the holomorphic tangent bundle). For simplicity, sometimes, the bundle projections will be denoted in the same way, when 
the base manifold is the same.\\ 
\indent 
Let $E$ be a vector bundle, endowed with a connection $\nabla$, over a manifold $M$. Suppose that we are given a 
morphism of vector bundles $\r:E\to TM$.\\ 
\indent 
Then, firstly, note that there exists a unique section $T$ of $TM\otimes\Lambda^2E^*$ such that 
\begin{equation} \label{e:gen_torsion} 
T(s_1,s_2)=\r\circ\bigl(\nabla_{\r\circ s_1}s_2-\nabla_{\r\circ s_2}s_1\bigr)-[\r\circ s_1,\r\circ s_2]\;, 
\end{equation} 
for any (local) sections $s_1$ and $s_2$ of $E$\,; we call $T$ the \emph{torsion (with respect to $\r$)} of $\nabla$.\\ 
\indent 
Let $F$ and $G$ be the typical fibre and structural group of $E$\,, respectively, and assume $\nabla$ compatible with $G$\,. 
Denote by $(P,M,G)$ the frame bundle of $E$ and let $\H\subseteq TP$ be the principal connection on $P$ corresponding to $\nabla$.\\ 
\indent 
On composing the projection 
$P\times F\to E$ with $\r$\,, we obtain a morphism of vector bundles from $P\times F$ to $TM$ which covers 
the projection $\p:P\to M$\,. 
Consequently, this morphism factorises as a morphism of vector bundles, over $P$, from $P\times F$ to $\H$ followed by the 
canonical morphism from $\H\bigl(=\p^*(TM)\bigr)$ onto $TM$. Thus, if $\xi\in F$ 
the corresponding (constant) section of $P\times F$ determines a horizontal vector field $B(\xi)$ on $P$.\\ 
\indent  
Note that, $B(\xi)$ is characterised by $\dif\!\p\bigl(B(\xi)_u\bigr)=\r(u\xi)$\,, for any $u\in P$, 
and the fact that it is horizontal (compare \cite[p.\ 119]{KoNo}\,). However, unlike the classical case 
$B(\xi)$ may have zeros; indeed $B(\xi)$ is zero at $u\in P$ if and only if $\r(u\xi)=0$\,. 
Also, $\H$ is generated as a vector bundle by all $B(\xi)$\,, $\xi\in F$, if and only if $\r$ is surjective.\\ 
\indent  
Furthermore, $B:F\to\G(TP)$ is $G$-equivariant. Indeed, if we denote by $R_a$ (the differential of) the right translation by some $a\in G$ 
on $P$, we have 
$$\dif\!\p\bigl(\bigl(R_a\bigl(B(\xi)\bigr)\bigr)_u\bigr)=\dif\!\p\bigl(B(\xi)_{ua^{-1}}\bigr)=\r\bigl(u(a^{-1}\xi)\bigr)\;.$$
Hence, $R_a\bigl(B(\xi)\bigr)=B(a^{-1}\xi)$\,, for any $a\in G$ and $\xi\in F$; in particular, $[A,B(\xi)]=B(A\xi)$ for any $A\in\mathfrak{g}$ 
and $\xi\in F$, where $\mathfrak{g}$ is the Lie algebra of $G$ and we denote in the same way its elements and the corresponding 
fundamental vector fields on $P$ (compare \cite[Proposition III.2.3]{KoNo}\,). 

\begin{rem}  
Let $\xi\in F$ and let $(u(t))_t$ be an integral curve of $B(\xi)$\,; denote $c=\p\circ u$\,, $s=u\xi$\,. Then $c$ is a curve 
in $M$, and $s$ is a section of $c^*E$ satisfying $\r\circ s=\dot{c}$ and $(c^*\nabla)(s)=0$\,.  
These curves $s$ lead to a natural generalization of the notion of geodesic of a connection on a manifold. Note that, for any 
$e\in E$ there exists a unique germ of such a curve $s$ with $s(0)=e$\,. 
\end{rem} 

\indent 
In this setting, Cartan's first structural equation is replaced by the following fact. 

\begin{prop} \label{prop:gen_torsion} 
For any $u\in P$ and $\xi,\e\in F$ we have 
\begin{equation} \label{e:gen_torsion_P} 
T(u\xi,u\e)=-\dif\!\p\bigl([B(\xi),B(\e)]_u\bigr)\;. 
\end{equation} 
\end{prop} 
\begin{proof} 
Let $u_0\in P$ and let $u$ be a local section of $P$, defined on some open neighbourhood $U$ of $x_0=\p(u_0)$\,, such that 
$u_{x_0}=u_0$ and the local connection form $\G$ of $\H$, with respect to $u$\,, is zero at $x_0$\,.\\ 
\indent  
If $\xi\in F$ then, under the isomorphism $P|_U=U\times G$ corresponding to $u$\,, we have $B(\xi)_{ua}=\bigl(\r(ua\xi),-\G\bigl(\r(ua\xi)\bigr)a\bigr)$\,, 
for any $a\in G$\,.\\  
\indent 
By using the fact that $\G_{x_0}=0$\,, we quickly obtain that, at $u_0$\,, both sides of \eqref{e:gen_torsion_P} are equal to 
$-[\r(u\xi),\r(u\e)]_{x_0}$\,, for any $\xi,\e\in F$.  
\end{proof} 

\indent 
Also, we obtain the following natural generalization of the first Bianchi identity. 

\begin{prop} \label{prop:Bianchi_1} 
Let $E$ be a vector bundle, over $M$, and suppose that there exists a morphism of vector bundles $\r:E\to TM$.\\ 
\indent 
Then the curvature form $R$ of any torsion free connection on $E$ satisfies  
\begin{equation} \label{e:Bianchi_1} 
\r\bigg(R\bigl(\r(e_1),\r(e_2)\bigr)e_3+R\bigl(\r(e_2),\r(e_3)\bigr)e_1+ R\bigl(\r(e_3),\r(e_1)\bigr)e_2\bigg)=0\;, 
\end{equation} 
for any $e_1\,,e_2\,,e_3\in E$\,. 
\end{prop}  
\begin{proof} 
Let $\O$ be the curvature form of the corresponding principal connection on the frame bundle $P$ of $E$ 
(we think of $\O$ as a two-form on $P$ with values in the Lie algebra of the structural group of $E$\,; see \cite{KoNo}\,). 
Equation \eqref{e:Bianchi_1} is equivalent to the following 
\begin{equation} \label{e:Bianchi_1_princ} 
B\bigg(\O_{u\!}\bigl(B(\xi),B(\e)\bigr)\mu+\O_{u\!}\bigl(B(\e),B(\mu)\bigr)\xi+\O_{u\!}\bigl(B(\mu),B(\xi)\bigr)\e\bigg)_u=0\;, 
\end{equation}  
for any $u\in P$ and $\xi$\,, $\e$\,, $\mu$ in the typical fibre $F$ of $E$\,.\\ 
\indent 
By using the fact that the connection is torsion free, we obtain that, for any $u\in P$ and $\xi,\e,\mu\in F$, 
the horizontal part of $\bigl[B(\mu),[B(\xi),B(\e)]\bigr]_u$ is $B\bigl(\O_{u\!}\bigl(B(\xi),B(\e)\bigr)\mu\bigr)_u$\,. 
Therefore \eqref{e:Bianchi_1_princ} is just the horizontal part, at $u$\,, of the Jacobi identity, for the usual bracket,  
applied to $B(\xi)$\,, $B(\e)$\,, $B(\mu)$\,. 
\end{proof} 

\indent 
Let $S$ be a submanifold of a Grassmannian of $F$ on which 
$G$ acts transitively. Then $Z=P\times_GS$ is a subbundle of a Grassmannian bundle of $E$ on which $\nabla$ 
induces a connection $\H\subseteq TZ$\,.\\ 
\indent
Suppose that for any $p\in Z$ the restriction of $\r$ to $p$ is an isomorphism 
onto some vector subspace of $T_{\p(p)}M$, where $\p:Z\to M$ is the projection. 
Then we can construct a distribution $\Cal$ on $Z$ by requiring $\Cal\subseteq\H$ and $\dif\!\p(\Cal_p)=\r(p)$\,, for any $p\in Z$\,.  

\begin{prop} \label{prop:rho_integrab} 
The following assertions are equivalent, where $R$ and $T$ are the curvature form and the torsion of $\nabla$, respectively:\\ 
\indent 
\quad{\rm (i)} $\Cal$ is integrable;\\ 
\indent 
\quad{\rm (ii)} $R\bigl(\Lambda^{2\!}\bigl(\r(p)\bigr)\bigr)(p)\subseteq p$ and 
$T\bigl(\Lambda^2p\bigr)\subseteq\r(p)$\,, for any $p\in Z$\,.\\   
\indent 
Consequently, if\/ $\nabla$ is torsion free and $\Cal$ is integrable then  
\begin{equation} \label{e:first_Bianchi} 
R\bigl(\r(e_1),\r(e_2)\bigr)e_3+R\bigl(\r(e_2),\r(e_3)\bigr)e_1+ R\bigl(\r(e_3),\r(e_1)\bigr)e_2=0\;, 
\end{equation} 
for any $p\in Z$ and $e_1,e_2,e_3\in p$\,.  
\end{prop} 
\begin{proof} 
Let $(P,M,G)$ be the frame bundle of $E$\,, and let $H$ be the isotropy subgroup of $G$ at some $p_0\in S$\,. 
Then $Z=P/H$, and let $\s:P\to Z$ be the projection.\\ 
\indent 
If we denote $\Cal_P=(\dif\!\s)^{-1}(\Cal)$ then, as $\s$ is a surjective submersion, we also have $\Cal=(\dif\!\s)(\Cal_P)$\,. 
Therefore $\Cal$ is integrable if and only if $\Cal_P$ is integrable.\\ 
\indent 
Now, note that $\Cal_P$ is generated by all $B(\xi)$\,, with $\xi\in p_0$ and all (the fundamental vector fields) $A\in\mathfrak{h}$\,, 
where $\mathfrak{h}$ is the Lie algebra of $H$. Hence, $\Cal_P$ is integrable if and only if $[B(\xi),B(\e)]$ is a section of $\Cal_P$\,, 
for any $\xi,\e\in p_0$\,; equivalently, $\O_{u\!}\bigl(B(\xi),B(\e)\bigr)\in\mathfrak{h}$ and $\dif\!\p\bigl([B(\xi),B(\e)]_u\bigr)\in\r(up_0)$
for any $u\in P$ and $\xi,\e\in p_0$\,. Together with Cartan's second structural equation and Proposition \ref{prop:gen_torsion}\,, this 
completes the proof. 
\end{proof} 

\indent 
Note that, the last statement of Proposition \ref{prop:rho_integrab} could have been proved directly by observing that, under that hypothesis,  
the leaves of $\Cal$ are, locally, projected by $\p$ onto submanifolds of $M$ on which $\nabla$ induces a torsion free connection.\\ 
\indent 
In the following definition, the notations are as in Proposition \ref{prop:rho_integrab}\,. 

\begin{defn} \label{defn:first_Bianchi} 
1) We say that $\nabla$ \emph{satisfies the first Bianchi identity} if it is torsion free 
and \eqref{e:first_Bianchi} holds for any $e_1,e_2,e_3\in E$\,.\\ 
\indent 
2) We say that $\nabla$ \emph{satisfies the first Bianchi identity, with respect to $Z$\,,} if it is torsion free 
and \eqref{e:first_Bianchi} holds for any $p\in Z$ and $e_1,e_2,e_3\in p$\,. 
\end{defn}  

\indent 
Let $E$ be a vector bundle, endowed with a connection $\nabla$, over a manifold $M$. Suppose that $\r:E\to TM$ 
is a morphism of vector bundles and let $[s_1,s_2]=\nabla_{\r\circ s_1}s_2-\nabla_{\r\circ s_2}s_1$\,, 
for any sections $s_1$ and $s_2$ of $E$\,. 
  
\begin{prop} \label{prop:Lie_algebroid}
The following assertions are equivalent:\\ 
\indent 
{\rm (i)} $\nabla$ satisfies the first Bianchi identity;\\ 
\indent 
{\rm (ii)} $(E,[\cdot,\cdot],\r)$ is a Lie algebroid. 
\end{prop} 
\begin{proof} 
This is a straightforward computation. 
\end{proof}

\section{On the integrability of the co-CR quaternionic structures}  \label{section:co-cr_q_integrab} 

\indent 
A \emph{quaternionic vector bundle} is a vector bundle $E$ whose structural group is the Lie group ${\rm Sp}(1)\cdot{\rm GL}(k,\Hq)$ 
acting on $\Hq^{\!k}$ by $\bigl(\pm(a,A),q\bigr)\mapsto aqA^{-1}$, for any $\pm(a,A)\in{\rm Sp}(1)\cdot{\rm GL}(k,\Hq)$ 
and $q\in\Hq^{\!k}$. Then the morphism of Lie groups ${\rm Sp}(1)\cdot{\rm GL}(k,\Hq)\to{\rm SO}(3)$\,, 
$\pm(a,A)\to\pm a$\,, induces an oriented Riemannian vector bundle of rank three whose sphere bundle $Z$ 
is the bundle of \emph{admissible linear complex structures} of $E$.\\ 
\indent 
An \emph{almost co-CR quaternionic structure} on $M$ is a pair $(E,\r)$ where $E$ is a quaternionic vector bundle over $M$ 
and $\r:E\to TM$ is a surjective morphism of vector bundles whose kernel contains no nonzero subspace preserved by some admissible 
linear complex structure of $E$.\\ 
\indent 
By duality, we obtain the notion of \emph{almost CR quaternionic structure}.\\ 
\indent 
Let $(M,E,\r)$ be an almost co-CR quaternionic manifold and let $\nabla$ be a compatible connection on $E$ (that is, 
$\nabla$ is compatible with the structural group of $E$). Then we can construct a complex distribution $\Cal$ on the bundle $Z$ 
of admissible linear complex structures of $E$, as follows. Firstly, for any $J\in Z$, let $\mathcal{B}_J$ be the horizontal lift 
of $\r\bigl({\rm ker}(J+{\rm i})\bigr)$\,, with respect to $\nabla$. Then $\Cal=({\rm ker}\dif\!\p)^{0,1}\oplus\mathcal{B}$ 
is a complex distribution on $Z$ such that $\Cal+\overline{\Cal}=T^{\C\!}Z$\,; that is, $\Cal$ is an \emph{almost co-CR structure} 
on $Z$.\\ 
\indent 
We say that $(M,E,\r,\nabla)$ is \emph{co-CR quaternionic} if $\Cal$ is integrable. Note that, then $\Cal\cap\overline{\Cal}$ is the complexification 
of (the tangent bundle of) a foliation $\Fa$ on $Z$\,; moreover, with respect to it, $\Cal$ is projectable onto complex structures on the local leaf spaces 
of $\Fa$. If there exists a surjective submersion $\p_Z:Z\to Y$ such that ${\rm ker}\dif\!\p_Y=\Fa$ and $\Cal$ is projectable with respect to $\p_Y$ 
(the latter condition is unnecessary if the fibres of $\p_Y$ are connected) 
then the complex manifold $\bigl(Y,\dif\!\p_Y(\Cal)\bigr)$ is the \emph{twistor space} of $(M,E,\r,\nabla)$\,. 
\begin{displaymath}
\xymatrix{
          &  Z \ar[dl]_{\p_Y}   \ar[dr]^{\p}  &   \\
      Y   &                       & M
     }
\end{displaymath} 

\indent 
Note that, if $\r$ is an isomorphism then we obtain the classical notion of \emph{quaternionic manifold} \cite{Sal-dg_qm} 
(see \cite[Remark 2.10(2)]{IMOP}\,). Also, more information on (co-)CR quaternionic manifolds can be found 
in \cite{fq}, \cite{fq_2}, \cite{MP2}, \cite{Pan-twistor_(co-)cr_q}\,.  
 
\begin{thm} \label{thm:co-cr_q_integrab} 
Let $(M,E,\r)$ be an almost co-CR quaternionic manifold and let $Z$ be the bundle of admissible linear complex structures on $E$\,. 
Let $\nabla$ be a compatible connection on $E$\,.\\ 
\indent 
The following assertions are equivalent, where $R$ and $T$ are the curvature form and the torsion of $\nabla$, respectively:\\ 
\indent 
\quad{\rm (i)} $(E,\r,\nabla)$ is integrable;\\ 
\indent 
\quad{\rm (ii)} $R\bigl(\Lambda^{2\!}\bigl(\r\bigl(E^J\bigr)\bigr)\bigr)\bigl(E^J\bigr)\subseteq E^J$ and 
$T\bigl(\Lambda^2\bigl(E^J\bigr)\bigr)\subseteq\r\bigl(E^J\bigr)$, for any $J\in Z$\,, where $E^J={\rm ker}(J+{\rm i})$\,.
\end{thm}  
\begin{proof} 
Assuming real-analyticity, this follows quickly from Proposition \ref{prop:rho_integrab}\,, after a complexification. 
In the smooth category, this follows from an extension of Proposition \ref{prop:rho_integrab}\,, similar to \cite[Theorem A.3]{fq}\,. 
\end{proof} 

\indent 
Let $M$ be a quaternionic manifold, $\dim M=4k$\,, endowed with a torsion free compatible connection. 
Denote by $L$ the complexification of the line bundle over $M$ characterised by the fact that its $4k$ tensorial power is $\Lambda^{4k}TM$ 
(we use the orientation on $M$ compatible with all of the admissible linear complex structures on it).\\ 
\indent 
Then, at least locally, we have $T^{\C\!}M=H\otimes W$, where $H$ and $W$ are complex vector bundles of rank $2$ and $2k$\,, respectively, 
and the structural group of $H$ is ${\rm SL}(2,\C\!)$ ($H$ and $W$ exist globally if and only if 
the vector bundle generated by the admissible linear complex structures on $M$ is spin).\\ 
\indent 
Denote $H'=(L^*)^{k/k+1}\otimes H$. Then $H'\setminus0$ is endowed with a natural hyper-complex structure 
(\cite{Sal-dg_qm}\,; see \cite{PePoSw-98}\,), such that the projection onto $M$ is twistorial. In particular, 
on endowing $H'\setminus0$ with one of the admissible complex structures (corresponding 
to some imaginary quaternion of length $1$\,) then $H'\setminus0$ is the total space of a holomorphic principal bundle over the 
twistor space $Z$ of $M$, with group $\C\!\setminus\{0\}$\,. We shall denote by $\mathcal{L}$ the dual of the corresponding holomorphic 
line bundle over $Z$\,; note that, if $m$ is even then $\mathcal{L}^m$ is globally defined. 
For example, if $M=\Hq\!P^k$ then $\mathcal{L}$ is just the hyperplane line bundle over $\C\!P^{2k+1}$.\\ 
\indent 
Now, let $U_m=\odot^m(H')^*$, where $\odot$ denotes the symmetric product, $m\in\mathbb{N}$\,.\\ 
\indent 
If $m$ is even then $U_m$ is globally defined and is the complexification of a (real) 
vector bundle which will be denote in the same way (note that, $L^{k/k+1}\otimes U_2$ is just the oriented Riemannian vector bundle of 
rank three generated by the admissible linear complex structures on $M$). 
Let $F$ be a vector bundle over $M$ endowed with a connection whose $(0,2)$ components of its curvature, 
with respect to any admissible linear complex structure on $M$, are zero. 
We endow $\mathcal{F}=(\p^*F)^{\C\!}$ with the (Koszul--Malgrange) holomorphic structure determined by the pull back of the connection 
on $F$ and the complex structure of $Z$.\\ 
\indent 
If $m$ is odd then $U_m$ is a hypercomplex vector bundle. Therefore if $F$ is a hypercomplex vector bundle over $M$ then $U_m\otimes F$ 
is the complexification of a vector bundle which will be denoted in the same way; in the tensor product $U_m$ and $F$ are endowed with $I_1$ and $J_1$\,, respectively, 
whilst the conjugation on $U_m\otimes F$ is $I_2\otimes J_2$, where $I_i$ and $J_i$\,, $i=1,2,3$\,, 
give the linear hypercomplex structures of $U_m$ and $F$, respectively. 
Suppose that $F$ is endowed with a compatible connection whose $(0,2)$ components 
of its curvature, with respect to any admissible linear complex structure on $M$, are zero.  
On endowing $F$ with $J_1$\,, let $\mathcal{F}=\p^*F$ endowed with the holomorphic structure determined by the pull back of the connection 
on $F$ and the complex structure of $Z$.\\ 
\indent 
If $m=1$ the next result gives \cite[Theorem 7.2]{Sal-dg_qm}\,. 

\begin{thm} \label{thm:holo_bundles_over_Z} 
{\rm (a)} There exists a natural co-CR quaternionic structure on the total space of\/ $U_m\otimes F$ whose twistor space is $\mathcal{L}^m\otimes\Fa$.\\ 
\indent 
{\rm (b)} Conversely, let $\mathcal{Z}$ be a holomorphic vector bundle over $Z$ such that:\\ 
\indent 
\quad{\rm (i)} the Birkhoff--Grothendieck decomposition of $\mathcal{Z}$ restricted to each twistor line  
contains only terms of Chern number $m$\,;\\ 
\indent 
\quad{\rm (ii)} $\mathcal{Z}$ is endowed with a conjugation covering the conjugation of $Z$\,.\\ 
\indent 
Then $\mathcal{Z}$ is the twistor space of a co-CR quaternionic manifold, obtained as in {\rm (a)}\,. 
\end{thm} 
\begin{proof} 
For simplicity, we work in the complex-analytic category. Thus, in particular, at least locally, a (complex-)quaternionic vector bundle 
is a bundle which is the tensor product of a vector bundle of rank $2$ and another vector bundle; for example, on denoting $V=L^{k/k+1}\otimes W$,  
we have $TM=U_1^*\otimes V$.\\ 
\indent 
Also, let $E=(U_{m-2}\oplus U_m)\otimes F$\,. As $U_{m-2}\oplus U_m=U_1\otimes U_{m-1}$\,, we have $E=U_1\otimes U_{m-1}\otimes F$, 
and, in particular, $E$ is a quaternionic vector bundle. 
Furthermore, by using the fact that the structural group of $L^{k/k+1}\otimes U_1^*$ 
is ${\rm SL}(2,\C\!)$\,, we obtain that $U_1=U_1^*\otimes L^{2k/k+1}$. Hence, also, $E\oplus TM$ is a quaternionic vector bundle.\\ 
\indent 
By using the induced connection on $U_m\otimes F$ we obtain 
\begin{equation} \label{e:for_holo_bundles_over_Z}
\begin{split} 
T(U_m\otimes F)&=\p^*(U_m\otimes F)\oplus\p^*(TM)\;,\\ 
\p^*(E\oplus TM)&=T(U_m\otimes F)\oplus\p^*(U_{m-2}\otimes F)\;, 
\end{split} 
\end{equation} 
where $\p:U_m\otimes F\to M$ is the projection.\\ 
\indent 
Thus, $\p^*(E\oplus TM)$ and the projection $\r$ from it onto $T(U_m\otimes F)$ provide 
an almost co-CR quaternionic structure on $U_m\otimes F$. Furthermore, the connections on $M$ and $F$ induce a compatible connection $\nabla$ on 
$\p^*(E\oplus TM)$\,, which preserves the decomposition given by the second relation of \eqref{e:for_holo_bundles_over_Z}\,. 
Thus, $\nabla$ is flat when restricted to the fibres of $U_m\otimes F$, whilst if $X\in\p^*(TM)$ then $\nabla_X$ is given by the pull back of the connection 
on $E\oplus TM$; in particular, if $X$ and $Y$ are pull backs of local vector fields on $M$ then $\nabla_XY$ is the pull back of 
the covariant derivative of $\dif\!\p(Y)$ along $\dif\!\p(X)$\,.\\ 
\indent 
We have $U_m=\odot^mU_1$\,, where $\odot$ denotes the symmetric product. 
Also, each $e\in U_1$ may be extended to a covariantly constant local section 
of $U_1$ (this is the reason for which the `tensorisation' with $(L^*)^{k/k+1}$ is needed).\\ 
\indent 
In this setting, the bundle of admissible linear complex structures on $M$ is replaced by $P(U_1^*)$ so that if $J$ `corresponds' to $[e]$\,, for some $e\in U_1^*$,  
then ${\rm ker}(J+{\rm i})$ corresponds to the space $\{e\otimes f\,|\,f\in V\,\}$\,. Then, on denoting $\mathcal{E}=\p^*(E\oplus TM)$\,, 
for any nonzero $e\in U_1$\,, we have that $\r(\mathcal{E}^e)$ is isomorphic to the direct sum of $\{e\otimes f\,|\,f\in V\,\}$ 
and the tensor product of the corresponding fibre of $F$ with the space of polynomials from $U_m=\odot^mU_1$ which are divisible by $e$\,.\\ 
\indent 
To verify that condition (ii) of Theorem \ref{thm:co-cr_q_integrab} is satisfied we shall, also, use the fact that $\nabla$ restricted to each fibre of $U_m\otimes F$ 
is flat. This and the fact that $M$ is quaternionic (and endowed with a torsion free connection) quickly implies 
that the curvature form of $\nabla$ satisfies (ii) of Theorem \ref{thm:co-cr_q_integrab}\,. 
For the torsion $T$, it is sufficient to check the condition on pairs of local sections $A,X$ and $X,Y$ from $\mathcal{E}^e$ 
with $A$ induced by a section of $E$ and $X,Y$ induced by sections of $TM$, where $e\in U_1$\,. Then we have $T(A,X)=-\r(\nabla_XA)$ 
and $T(X,Y)$ is the `vertical' component of $-[X,Y]$\,; in particular, $T(X,Y)$ is determined by the curvature form of $U_m\otimes F$, 
applied to $(X,Y)$\,.\\ 
\indent 
Locally, we may assume $L$ trivial so that $U_1=U_1^*$ but, note that, this isomorphism does not preserve the connections 
(the connection on $U_1$ is just the dual of the connection on $U_1^*$). 
Then we may choose $(e_1,e_2)$ a local frame for $U_1^*$ such that it corresponds to $(e^2,-e^1)$\,, where $(e^1,e^2)$ is the dual 
of $(e_1,e_2)$\,, and such that $e_1$ is covariantly constant. Thus, we have to check that $T(A,X)$ and $T(X,Y)$ are contained by $\r(\mathcal{E}^{e_1})$\,, 
where $A$ is the pull back of the tensor product of a local section of $F$ and a polynomial of degree $m$ which is divisible by $e^2$, whilst 
$X=\p^*(e_1\otimes u)$ and $Y=\p^*(e_1\otimes v)$\,, with $u$ and $v$ local sections of $V$. Now, the condition on the torsion follows quickly 
by using the fact that $e_1$ is covariantly constant and the fact that the curvature form of $F$ is zero when restricted to spaces of the form 
$\{e\otimes f\,|\,f\in V\,\}$\,, with $e\in U_1^*$.\\  
\indent 
In the complex-analytic category, the twistor space of $M$ is (locally) the leaf space of the foliation $\V$ on $P(U_1)$ which, at each $[e]\in P(U_1)$\,, 
is the horizontal lift of the space $\{e\otimes f\,|\,f\in V\,\}$\,. Similarly, the twistor space of $U_m\otimes F$ is the leaf space of the foliation 
$\V_m$ on $\p^*\bigl(P(U_1)\bigr)$ 
which at each $\p^*[e]$ is the horizontal lift of $\r(\mathcal{E}^e)$\,.\\ 
\indent 
On the other hand, the pull back of $\mathcal{L}^*\setminus0$ to $P(U_1)$ is the principal bundle whose projection is $U_1\setminus0\to P(U_1)$\,;  
equivalently, the pull back of $\mathcal{L}^*$ to $P(U_1)$ is the tautological line bundle over $P(U_1)$\,.   
This is, further, equivalent to the fact that the pull back of $\mathcal{L}$ to $P(U_1)$ is (locally; globally, if $H^1(M,\C\!\setminus\{0\}\,)$ is zero) 
isomorphic to the quotient of $\p_1^*(U_1)$ through the tautological line bundle over $P(U_1)$\,, 
where $\p_1:U_1\to M$ is the projection. Therefore $\mathcal{L}$ is the leaf space of the foliation on $\p_1^*\bigl(P(U_1)\bigr)$ 
which at each $\p_1^*[e]$ is the horizontal lift of $[e]\oplus\{e\otimes f\,|\,f\in V\,\}$\,.  
Similarly, we obtain that the twistor space of $U_m$ is $\mathcal{L}^m$. Together with $\p_m^*\bigl(P(U_1)\bigr)=U_m+P(U_1)$\,,  
this gives a surjective submersion $\phi_m:U_m+P(U_1)\to\mathcal{L}^m$ which is linear along the fibres of the projection from 
$U_m+P(U_1)$ onto $P(U_1)$\,; that is, $\phi_m$ is a surjective morphism of vector bundles, covering the surjective submersion $P(U_1)\to Z$\,.\\ 
\indent 
Also, the condition on the connection of $F$ is equivalent to the fact that its pull back to $P(U_1)$ is flat when restricted to the leaves of $\V$. Hence, 
the pull back of $F$ to $P(U_1)$ is, also, the pull back of a vector bundle $\Fa$ on $Z$\,. Thus, we, also, have a surjective morphism of vector bundles 
$\phi:F+P(U_1)\to\mathcal{F}$\,, covering $P(U_1)\to Z$\,.\\ 
\indent 
Therefore there exists a morphism of vector bundles $\psi$ from $(U_m\otimes F)+P(U_1)$ onto $\mathcal{L}^m\otimes\mathcal{F}$, 
covering $P(U_1)\to Z$\,. Moreover, ${\rm ker}\dif\!\psi=\V_m$ and, hence, the twistor space of $U_m\otimes F$ is~$\mathcal{L}^m\otimes\mathcal{F}$.\\
\indent 
Conversely, if $\mathcal{Z}$ is a vector bundle over $Z$ satisfying (i) then $\bigl(\mathcal{L}^*\bigr)^m\otimes\mathcal{Z}$ 
restricted to each twistor line is trivial. Thus, it corresponds (through the Ward transform) to a vector bundle $F$ over $M$ endowed with a connection 
whose curvature form is zero when restricted to spaces of the form $\{e\otimes f\,|\,f\in V\,\}$\,, with $e\in U_1$\,. 
Similarly to above, we obtain that $\mathcal{Z}$ is the twistor space of $U_m\otimes F$, and the proof is complete. 
\end{proof}     

\indent 
With the same notations as in Theorem \ref{thm:holo_bundles_over_Z}\,, the projection from $U_m\otimes F$ onto $M$ is the twistorial map 
corresponding to the projection from  $\mathcal{L}^m\otimes\Fa$ onto $Z$. Also, further examples of co-CR quaternionic manifolds can be obtained 
by taking direct sums of bundles $U_m\otimes F$ (with different values for $m$).\\ 
\indent 
Here is another application of Theorem \ref{thm:co-cr_q_integrab}\,. 

\begin{cor} \label{cor:co-cr_q_integrab} 
Let $(M,E,\r)$ be an almost co-CR quaternionic manifold, $\rank E>4$\,, and let $Z$ be the bundle of admissible linear complex structures 
of $E$\,.\\ 
\indent 
If there exists a compatible connection $\nabla$ on $E$ which satisfies the first Bianchi identity, with respect to $Z$\,, 
then $(E,\r,\nabla)$ is integrable.
\end{cor} 
\begin{proof}  
Locally, we may suppose $E^{\C\!}=H\otimes F$, where $H$ and $F$ are complex vector bundles with $\rank H=2$\,. Moreover, 
the following hold:\\ 
\indent 
\quad(1) $Z=PH$ such that if $J\in Z$ corresponds to $[e]\in PH$ then $$E^J=\{e\otimes f\,|\,f\in F_{\p(e)}\}\;,$$ where $\p$ is the projection.\\ 
\indent 
\quad(2) $\nabla=\nabla^H\otimes\nabla^F$ for some connections $\nabla^H$ on $H$ and $\nabla^F$ on $F$.\\ 
\indent 
By Theorem \ref{thm:co-cr_q_integrab}\,, we have to show that, for any $e\in H$ and $f_1,f_2,f_3\in F$, we have 
$$R\bigl(\r(e\otimes f_1),\r(e\otimes f_2)\bigr)(e\otimes f_3)\in\{e\otimes f\,|\,f\in F_{\p(e)}\}\;;$$
equivalently, $R^E\bigl(\r(e\otimes f_1),\r(e\otimes f_2)\bigr)e$ is proportional to $e$\,, where $R^E$ is the curvature form of $\nabla^E$.\\ 
\indent 
We know that, for any $e\in H$ and $f_1,f_2,f_3\in F$, we have 
$$R\bigl(\r(e\otimes f_1),\r(e\otimes f_2)\bigr)(e\otimes f_3)+\textrm{circular\;permutations}=0\;,$$ 
which implies  
\begin{equation} \label{e:for_integrab_from_Bianchi} 
\bigl(R^E\bigl(\r(e\otimes f_1),\r(e\otimes f_2)\bigr)e\bigr)\otimes f_3+\textrm{circular\;permutations}\in\{e\otimes f\,|\,f\in F_{\p(e)}\}\;. 
\end{equation} 
\indent 
As $\rank E>4$\,, we have $\rank F>2$\,. Therefore \eqref{e:for_integrab_from_Bianchi} holds, for any $e\in H$ and $f_1,f_2,f_3\in F$, 
if and only if each term of the left hand side of \eqref{e:for_integrab_from_Bianchi} is contained by $\{e\otimes f\,|\,f\in F_{\p(e)}\}$\,, 
for any $e\in H$ and $f_1,f_2,f_3\in F$. The proof is complete. 
\end{proof} 

\indent 
Note that, in the proof of Corollary \ref{cor:co-cr_q_integrab} it is not used the fact that $\rank H=2$\,. 

\begin{prop} \label{prop:co-cr_q_with_first_Bianchi} 
Let $(M,E,\r)$ be an almost co-CR quaternionic manifold such that $\rank E>4$ and there exists a compatible connection 
$\nabla$ on $E$ which satisfies the first Bianchi identity.\\ 
\indent  
Then, locally, ${\rm ker}\r$ can be endowed with a quaternionic structure such that the projection from ${\rm ker}\r$ onto $M$ is a 
twistorial map. 
\end{prop}    
\begin{proof} 
{}From Proposition \ref{prop:Lie_algebroid} and \cite[Theorem 2.2]{KMa-BLMS95} it follows that, locally, there exists 
a section $\iota:TM\to E$ of $\r$ such that for any sections $s_1$ and $s_2$ of the vector subbundle ${\rm im}\,\iota\subseteq E$ 
we have that $\nabla_{\r\circ s_1}s_2-\nabla_{\r\circ s_2}s_1$ is a section of ${\rm im}\,\iota$\,. In particular, we have 
$E={\rm ker}\r\oplus TM$, where we have used the obvious isomorphism $TM={\rm im}\,\iota$\,.\\ 
\indent 
Furthermore, Proposition \ref{prop:Lie_algebroid} quickly implies that $\nabla$ restricts to a flat connection $\nabla^v$ 
on ${\rm ker}\r$\,. Locally, we may suppose that $\nabla^v$ is the trivial connection corresponding to some trivialization 
of ${\rm ker}\r$\,; that is, ${\rm ker}\r$ is generated by (global) sections which are covariantly constant, with respect to $\nabla$.\\ 
\indent 
Let $\p:{\rm ker}\r\to M$ be the projection. Note that, we have two decompositions $\p^*E=\p^*({\rm ker}\r)\oplus\p^*(TM)$ 
and $T({\rm ker}\r)=\p^*({\rm ker}\r)\oplus\p^*(TM)$\,, where the latter is induced by $\nabla^v$. Therefore we have a vector 
bundle isomorphism $T({\rm ker}\r)=\p^*E$ which depends only of $\iota$ (and the given co-CR quaternionic structure). 
Hence, ${\rm ker}\r$ is endowed with an almost quaternionic structure.\\  
\indent 
To complete the proof it is sufficient to prove that $\p^*\nabla$ is torsion free. Indeed, let $U,V$ be sections of $\p^*({\rm ker}\r)$ 
induced by sections of ${\rm ker}\r$ which are covariantly constant, with respect to $\nabla$, and let $X,Y$ be sections of 
$\p^*(TM)$ induced by vector fields on $M$; in particular, $X,Y$ are projectable, with respect to $\dif\!\p$\,.\\ 
\indent 
Then we have that all of $[U,V]$\,, $[U,X]$\,, $(\p^*\nabla)_UV$, $(\p^*\nabla)_VU$\,, $(\p^*\nabla)_UX$\,, $(\p^*\nabla)_XU$ are zero. 
Also, as $\nabla$ is torsion free, we have $(\p^*\nabla)_XY-(\p^*\nabla)_YX-[X,Y]=0$\,, thus, completing the proof. 
\end{proof}   

\indent 
Let $N$ be a quaternionic-K\"ahler manifold and let $\nabla$ be its Levi--Civita connection. If $M\subseteq N$ is a hypersurface or a CR quaternionic submanifold 
\cite{MP2} then the following assertions are equivalent:\\ 
\indent 
\quad(i) $\nabla$ restricted to $TN|_M$ satisfies the first Bianchi identity;\\ 
\indent 
\quad(ii) $M$ is geodesic and the normal connection is flat.

\section{On the integrability of the $f$-quaternionic structures} \label{section:f-q_integrab} 

\indent 
An \emph{almost $f$-quaternionic structure} \cite{fq_2} on a manifold $M$ is a pair $(E,V)$\,, where $E$ is a quaternionic vector bundle over $M$, 
with $V,TM\subseteq E$ vector subbundles such that $E=V\oplus TM$ and $JV\subseteq TM$, for any admissible linear complex structure 
$J$ of~$E$. Then $(E,\iota)$ and $(E,\r)$ are almost CR quaternionic and almost co-CR quaternionic structures  
on $M$, where $\iota:TM\to E$ and $\r:E\to TM$ are the inclusion and the projection, respectively.\\ 
\indent 
Any almost $f$-quaternionic structure on $M$ corresponds to a reduction of its frame bundle to the group $G_{l,m}$ 
of \emph{$f$-quaternionic linear isomorphisms} of $({\rm Im}\Hq\!)^l\times\Hq^{\!m}$, in particular $\dim M=3l+4m$\,. 
More precisely, $G_{l,m}={\rm GL}(l,\R)\times\bigl({\rm Sp}(1)\cdot{\rm GL}(m,\Hq)\bigr)$\,, 
where ${\rm Sp}(1)\cdot{\rm GL}(m,\Hq)$ acts canonically on $\Hq^{\!m}$, whilst the action of $G_{l,m}$ on $({\rm Im}\Hq\!)^l=\R^l\otimes{\rm Im}\Hq$ 
is given by the tensor product of the canonical representations of ${\rm GL}(l,\R)$ and ${\rm SO}(3)$ on $\R^l$ and ${\rm Im}\Hq$, 
respectively, and the canonical morphisms of Lie groups from $G_{l,m}$ onto ${\rm GL}(l,\R)$ and ${\rm SO}(3)$\,. 
Furthermore, $G_{l,m}$ is isomorphic to the group of quaternionic linear isomorphisms of $\Hq^{\!l+m}$ which preserve both $\R^l$ and 
$({\rm Im}\Hq\!)^l\times\Hq^{\!m}$.\\  
\indent 
Consequently, any almost $f$-quaternionic structure on $M$, also, corresponds to a decomposition $TM=(V\otimes Q)\oplus W$, 
where $V$ is a vector bundle, $Q$ is an oriented Riemannian vector bundle of rank three, and $W$ is a quaternionic vector bundle 
such that the frame bundle of $Q$ is the principal bundle induced by the frame bundle of $W$ through the canonical morphism of Lie groups 
${\rm Sp}(1)\cdot{\rm GL}(m,\Hq)\to{\rm SO}(3)$\,, where $\rank W=4m$\,.\\ 
\indent 
Thus, any connection $\nabla$ on $E$ compatible with $G_{l,m}$ induces a connection $D$ on $M$ 
such that $\nabla=D^V\oplus D$, where $D^V$ is the connection induced by $D$ on $V$; then we say that $D$ 
is \emph{compatible with $(E,V)$}\,. Moreover, if we denote by $T^E$ and $T$ the torsions of $\nabla$ and $D$, 
respectively, then $T^E=\r^*T$; in particular, $\nabla$ is torsion free if and only if $D$ is torsion free.\\ 
\indent  
Furthermore, $D=\bigl(D^V\otimes D^Q\bigr)\oplus D^W$, where $D^W$ is a compatible connection on the quaternionic vector bundle $W$,  
and $D^Q$ is the connection induced by $D^W$ on $Q$. In particular, 
if $D$ is torsion free then $V\otimes Q$ and $W$ are foliations on $M$, 
and the leaves of the latter are quaternionic manifolds. 

\begin{cor} \label{cor:f-q_D} 
Let $(E,V)$ be an almost $f$-structure on $M$ and let $\nabla$ be the connection on $E$ induced 
by some connection $D$ on $M$, compatible with $(E,V)$\,. Let $\r:E\to TM$ be the projection, 
$T$ the torsion of $D$, and $R^{Q}$ the curvature form of the connection induced on $Q$\,.\\ 
\indent 
Then $(E,\r,\nabla)$ is integrable if and only if, for any $J\in Z$\,, we have 
\begin{equation} \label{e:f-q_D} 
\begin{split} 
&T\bigl(\Lambda^{2\!}\bigl(\r\bigl(E^J\bigr)\bigr)\bigr)\subseteq\r\bigl(E^J\bigr)\;,\\ 
&R^Q\bigl(\Lambda^{2\!}\bigl(\r\bigl(E^J\bigr)\bigr)\bigr)(J)\subseteq(J'+{\rm i}J'')^{\perp}\;, 
\end{split} 
\end{equation}  
where $E^J={\rm ker}(J+{\rm i})$\,, and $J',J''\in Z$ such that $(J,J',J'')$ is a positive orthonormal frame. 
\end{cor} 
\begin{proof} 
Let $D^V$ be the connection induced on $V$ and $T^E$ the torsion of $\nabla$. Because $T^E=\r^*T$, 
from Theorem \ref{thm:co-cr_q_integrab} we obtain that it is sufficient to prove that, for any $J\in Z$\,, 
the second relation of \eqref{e:f-q_D} holds if and only if  
\begin{equation} \label{e:f-q_D_1} 
R^E\bigl(\Lambda^{2\!}\bigl(\r\bigl(E^J\bigr)\bigr)\bigr)\bigl(E^J\bigr)\subseteq E^J\;, 
\end{equation} 
where $R^E$ is the curvature form of $\nabla$.\\ 
\indent 
We have $R^Q(X,Y)J=\bigl[R^E(X,Y),J\bigr]$\,, for any $J\in Z$ and $X,Y\in TM$.  
Therefore if $J\in Z$ and $A,B,C\in E^J$ then $$\bigl(R^Q\bigl(\r(A),\r(B)\bigr)J\bigr)C=-(J+{\rm i})\bigl(R^E\bigl(\r(A),\r(B)\bigr)C\bigr)\;.$$ 
As, up to a nonzero factor, $J+{\rm i}$ is the projection from $E^{\C}$ onto $\overline{E^J}$, we have that \eqref{e:f-q_D_1} holds 
if and only if $R^Q\bigl(\Lambda^{2\!}\bigl(\r\bigl(E^J\bigr)\bigr)\bigr)\bigl(E^J\bigr)=0$. But, for any $X,Y\in TM$, we have 
$R^Q(X,Y)J=\a(X,Y)(J'+{\rm i}J'')+\b(X,Y)(J'-{\rm i}J'')$\,, for some two-forms $\a$ and $\b$\,.\\ 
\indent 
To complete the proof just note that the obvious relation $J'+{\rm i}J''=J'\circ(1-{\rm i}J)$ 
and its conjugate imply that $(J'+{\rm i}J'')\bigl(E^J\bigr)=0$ 
whilst $J'-{\rm i}J''$ maps $E^J$ isomorphically onto $\overline{E^J}$.  
\end{proof}  

\begin{prop} \label{prop:f-q_rV>1_rW>0} 
Let $(E,V)$ be an almost $f$-structure on $M$ and let $\nabla$ be the connection on $E$ induced 
by some torsion free connection on $M$, compatible with $(E,V)$\,.\\ 
\indent 
If $\rank E>4\rank V\geq8$ then the connection induced on $Q$ is flat. 
\end{prop} 
\begin{proof} 
Let $TM=(V\otimes Q)\oplus W$ be the decomposition corresponding to $(E,V)$\,. Note that, 
$\rank E>4\rank V$ if and only if $\rank W>0$\,. Also, for any $J\in Z\,(\subseteq Q)$ and $U\in V$, 
we have $JU=U\otimes J$.\\ 
\indent 
Let $D$\,, $D^V$, $D^Q$, $D^W$ be the connections induced on $TM$, $V$, $Q$\,, $W$, 
and let $R^M$, $R^V$, $R^Q$, $R^W$ be their curvature forms, respectively;  
recall that, $D=\bigl(D^V\otimes D^Q\bigr)\oplus D^W$.\\ 
\indent 
Now, firstly, let $J\in Z$\,, $U\in V$ and $X,Y\in W$. From the first Bianchi identity applied to $D$ we obtain 
$$R^M(X,Y)(U\otimes J)+R^M(Y,U\otimes J)X+R^M(U\otimes J,X)Y=0\;;$$ 
hence, we, also, have  
\begin{equation} \label{e:f-q_rV>1_rW>0_1} 
\bigl(R^V(X,Y)U\bigr)\otimes J+U\otimes\bigl(R^Q(X,Y)J\bigr)=0\;. 
\end{equation} 
As $R^Q(X,Y)J\in J^{\perp}$, from \eqref{e:f-q_rV>1_rW>0_1} we obtain $R^Q(X,Y)J=0$\,.\\ 
\indent 
Secondly, let $J,J'\in Z$\,, $S,U\in V$, and $X\in W$. Then we have 
\begin{equation} \label{e:f-q_rV>1_rW>0_2} 
R^M(X,S\otimes J)(U\otimes J')+R^M(S\otimes J,U\otimes J')X+R^M(U\otimes J',X)(S\otimes J)=0\;. 
\end{equation}  
\indent 
Relation \eqref{e:f-q_rV>1_rW>0_2} implies $R^W(S\otimes J,U\otimes J')X=0$\,, and, as this holds for any $X\in W$, 
we obtain $R^Q(S\otimes J,U\otimes J')=0$\,.\\ 
\indent 
Furthermore, with $J=J'$, relation \eqref{e:f-q_rV>1_rW>0_2}\,, also, gives 
\begin{equation} \label{e:f-q_rV>1_rW>0_3} 
\begin{split} 
&\bigl(R^V(X,S\otimes J)U\bigr)\otimes J+U\otimes\bigl(R^Q(X,S\otimes J)J\bigr)\\ 
&+\bigl(R^V(U\otimes J,X)S\bigr)\otimes J+S\otimes\bigl(R^Q(U\otimes J,X)J\bigr)=0\;. 
\end{split} 
\end{equation} 
Thus, if in \eqref{e:f-q_rV>1_rW>0_3} we assume $S,U$ linearly independent, we obtain 
$R^Q(X,S\otimes J)J=0$\,, for any $X\in W$; equivalently, 
\begin{equation} \label{e:f-q_rV>1_rW>0_4} 
<R^Q(X,S\otimes J)J',J>=0\;, 
\end{equation}  
for any $X\in W$ and $J'\in Z$, orthogonal on $J$, where $<\cdot,\cdot>$ denotes the Riemannian structure on $Q$\,.\\ 
\indent 
Finally, if $(J,J',J'')$ is an orthonormal frame on $Q$\,, then \eqref{e:f-q_rV>1_rW>0_2} gives 
\begin{equation} \label{e:f-q_rV>1_rW>0_5} 
\begin{split} 
&\bigl(R^V(X,S\otimes J)U\bigr)\otimes J'+U\otimes\bigl(R^Q(X,S\otimes J)J'\bigr)\\ 
&+\bigl(R^V(U\otimes J',X)S\bigr)\otimes J+S\otimes\bigl(R^Q(U\otimes J',X)J\bigr)=0\;,  
\end{split} 
\end{equation}
for any $S,U\in V$ and $X\in W$. Hence, if $S,U$ are linearly independent, we deduce 
$<R^Q(X,S\otimes J)J',J''>=0$\,, for any $X\in W$. Together with \eqref{e:f-q_rV>1_rW>0_4}\,, 
this shows that $R^Q(X,S\otimes J)J'=0$\,, and the proof is complete. 
\end{proof}

\begin{prop} \label{prop:f-q_rV>2_rW=0} 
Let $(E,V)$ be an almost $f$-structure on $M$ and let $\nabla$ be the connection on $E$ induced 
by some torsion free connection on $M$, compatible with $(E,V)$\,; denote by $\r:E\to TM$ the projection.\\ 
\indent 
If $\rank E=4\rank V\geq12$ then $(E,\r,\nabla)$ is integrable. 
\end{prop}
\begin{proof} 
We shall use the same notations as in the proof of Proposition \ref{prop:f-q_rV>1_rW>0}\,. Note that, 
$\rank E=4\rank V\geq12$ if and only if $W=0$ and $\rank V\geq3$.\\ 
\indent 
Firstly, we shall prove that, for any $S,U\in V$ and any orthonormal frame $(J,J',J'')$ on $Q$\,, the 
folowing relations hold: 
\begin{equation} \label{e:f-q_rV>2_rW=0} 
\begin{split} 
R^Q(S\otimes J,U\otimes J)J=&\,0\;,\\ 
R^Q(S\otimes J,U\otimes J)J'=&\,0\;,\\ 
R^Q(S\otimes J,U\otimes J')J''=&\,0\;. 
\end{split} 
\end{equation} 
\indent 
{}From the first Bianchi identity applied to $D$ we obtain that, for any $J,J'\in Z$ and $S,T,U\in V$, we have: 
\begin{equation} \label{e:f-q_rV>2_rW=0_1} 
\begin{split} 
\bigl(R^V&(S\otimes J,T\otimes J)U\bigr)\otimes J'+U\otimes\bigl(R^Q(S\otimes J,T\otimes J)J'\bigr)\\ 
+&\bigl(R^V(T\otimes J,U\otimes J')S\bigr)\otimes J+S\otimes\bigl(R^Q(T\otimes J,U\otimes J')J\bigr)\\ 
&+\bigl(R^V(U\otimes J',S\otimes J)T\bigr)\otimes J+T\otimes\bigl(R^Q(U\otimes J',S\otimes J)J\bigr)=0\;. 
\end{split} 
\end{equation} 
\indent 
If in \eqref{e:f-q_rV>2_rW=0_1} we take $J=J'$ and $S,T,U$ linearly independent, we obtain that 
the first relation of \eqref{e:f-q_rV>2_rW=0} holds, for any $J\in Z$ and $S,U\in V$ (note that, if $S,U$ are linearly 
dependent then the first two relations of \eqref{e:f-q_rV>2_rW=0} are trivial).\\ 
\indent 
If in \eqref{e:f-q_rV>2_rW=0_1} we take $J\perp J'$ and either $S,T,U$ linearly independent, or $S=U$ and $S,T$ linearly independent, we obtain  
\begin{equation} \label{e:f-q_rV>2_rW=0_2} 
\begin{split} 
<R^Q(S\otimes J,U\otimes J)J',J''>\,=&\,0\;,\\ 
<R^Q(S\otimes J,U\otimes J')J,J''>\,=&\,0\;,  
\end{split} 
\end{equation}
for any orthonormal frame $(J,J',J'')$ on $Q$\,, and any $S,U\in V$.\\ 
\indent 
On swapping $J$ and $J'$, in the second relation of \eqref{e:f-q_rV>2_rW=0_2}\,, we deduce 
\begin{equation} \label{e:f-q_rV>2_rW=0_3} 
<R^Q(S\otimes J,U\otimes J')J',J''>\,=\,0\;, 
\end{equation}
for any orthonormal frame $(J,J',J'')$ on $Q$\,, and any $S,U\in V$.\\ 
\indent 
Now, the second relation of \eqref{e:f-q_rV>2_rW=0_2} and \eqref{e:f-q_rV>2_rW=0_3} imply that the third relation 
of \eqref{e:f-q_rV>2_rW=0} holds, as claimed.\\ 
\indent 
Further, the first relation of \eqref{e:f-q_rV>2_rW=0_2} implies that the second relation of \eqref{e:f-q_rV>2_rW=0} holds 
if and only if $<R^Q(S\otimes J,U\otimes J)J',J>\,=0$\,; but this is a consequence of the first relation of  \eqref{e:f-q_rV>2_rW=0}\,.\\ 
\indent 
To complete the proof, we use Corollary \ref{cor:f-q_D}\,. Thus, we have to prove that for any positive orthonormal frame $(J,J',J'')$ on $Q$\,, 
and any $S,U\in V$, the following holds: 
\begin{equation} \label{e:f-q_rV>2_rW=0_4} 
\begin{split} 
<R^Q(S\otimes J,U\otimes J)J,J'+{\rm i}J''>\,=&\,0\;,\\ 
<R^Q\bigl(S\otimes J,U\otimes(J'+{\rm i}J'')\bigr)J,J'+{\rm i}J''>\,=&\,0\;,\\ 
<R^Q\bigl(S\otimes(J'+{\rm i}J''),U\otimes(J'+{\rm i}J'')\bigr)J,J'+{\rm i}J''>\,=&\,0\;. 
\end{split} 
\end{equation} 
\indent 
Obviously, the first relation of \eqref{e:f-q_rV>2_rW=0_4} is an immediate consequence of the first relation of \eqref{e:f-q_rV>2_rW=0}\,.\\ 
\indent 
Note that, the second relation of \eqref{e:f-q_rV>2_rW=0_2} implies that, for any $A\in J^{\perp}\setminus\{0\}$, we have 
\begin{equation*}  
R^Q(S\otimes J,U\otimes A)J=\,\frac{<R^Q(S\otimes J,U\otimes A)J,A>}{<A,A>}\,A\;;    
\end{equation*}  
in particular, $<R^Q(S\otimes J,U\otimes A)J,A>\,=c<A,A>$\,, for any $A\in J^{\perp}$, where $c$ does not depend of $A$\,. 
As $J'+{\rm i}J''$ is isotropic this shows that the second relation of \eqref{e:f-q_rV>2_rW=0_4} holds.\\ 
\indent 
Finally, the last two relations of \eqref{e:f-q_rV>2_rW=0} (applied to suitable orthonormal frames) 
imply $R^Q\bigl(S\otimes(J'+{\rm i}J''),U\otimes(J'+{\rm i}J'')\bigr)J=0$\,.  Hence, also, the third relation of \eqref{e:f-q_rV>2_rW=0_4} holds.\\ 
\indent  
The proof is complete. 
\end{proof} 

\indent 
We can, now, give a new proof for \cite[Theorem 4.9]{fq_2}\,, where, note that, the condition $\rank V\neq1$ was 
discarded, due to a misprint. 

\begin{cor} \label{cor:f-q_D_torsion-free_integrab} 
Let $(E,V)$ be an almost $f$-structure on $M$ and let $\nabla$ be the connection on $E$ induced 
by some torsion free connection $D$ on $M$, compatible with $(E,V)$\,.\\ 
\indent 
If either $\rank E>4\rank V\geq8$ or $\rank E=4\rank V\geq12$ then $(E,V,\nabla)$ is integrable. 
\end{cor} 
\begin{proof} 
The integrability of the underlying almost co-CR quaternionic structure is a consequence of 
Corollary \ref{cor:f-q_D}\,, and Propositions \ref{prop:f-q_rV>1_rW>0} and \ref{prop:f-q_rV>2_rW=0}\,.\\ 
\indent 
The integrability of the undelying almost CR quaternionic structure is a consequence of 
Proposition \ref{prop:f-q_rV>1_rW>0} and \cite[Proposition 4.5]{fq}\,. 
\end{proof} 

\begin{cor} \label{for:connections_flat_on_Q_and_V} 
Let $(E,V)$ be an almost $f$-structure on $M$ and let $\nabla$ be the connection on $E$ induced 
by some torsion free connection on $M$, compatible with $(E,V)$\,.\\ 
\indent 
If the connection induced on $Q$ is flat then, also, the connection induced on $V$ is flat.\\ 
\indent 
The converse also holds if $\rank V\geq3$\,.
\end{cor} 
\begin{proof} 
As in the proof of Proposition \ref{prop:f-q_rV>1_rW>0} we deduce that \eqref{e:f-q_rV>1_rW>0_1}  
and \eqref{e:f-q_rV>1_rW>0_5} hold. Hence, if $R^Q=0$ then $R^V(X,Y)U=0$ and $R^V(S\otimes J,X)U=0$ for 
any $S,U\in V$\,, $X,Y\in W$ and $J\in Z$\,.\\ 
\indent 
{}From the first Bianchi identity applied to $D$ we obtain that, for any $J,J',J''\in Z$ and $S,T,U\in V$, we have: 
\begin{equation} \label{e:connections_flat_on_Q_and_V} 
\begin{split} 
\bigl(R^V&(S\otimes J,T\otimes J')U\bigr)\otimes J''+U\otimes\bigl(R^Q(S\otimes J,T\otimes J')J''\bigr)\\ 
+&\bigl(R^V(T\otimes J',U\otimes J'')S\bigr)\otimes J+S\otimes\bigl(R^Q(T\otimes J',U\otimes J'')J\bigr)\\ 
&+\bigl(R^V(U\otimes J'',S\otimes J)T\bigr)\otimes J'+T\otimes\bigl(R^Q(U\otimes J'',S\otimes J)J'\bigr)=0\;. 
\end{split} 
\end{equation} 
\indent 
If $R^Q=0$ and $J,J',J''$ are linearly independent then, from \eqref{e:connections_flat_on_Q_and_V} we obtain that 
$R^V(S\otimes A,T\otimes B)U=0$ for any $S,T,U\in V$ and $A,B\in Q$ (here, we have used the continuity of the map 
$(A,B)\mapsto R^V(S\otimes A,T\otimes B)U$, to allow $A,B$ linearly dependent).\\ 
\indent 
Similarly, if $\rank V\geq3$ and $R^V=0$\,, from Proposition \ref{prop:f-q_rV>1_rW>0} and \eqref{e:connections_flat_on_Q_and_V} 
we obtain $R^Q=0$\,. 
\end{proof} 

\indent 
We end with the following result. 

\begin{thm} \label{thm:f-q_generic_dims} 
Let $(E,V)$ be an almost $f$-structure on $M$ and let $\nabla$ be the connection on $E$ induced 
by some torsion free connection on $M$, compatible with $(E,V)$\,.\\ 
\indent 
If $\rank E>4\rank V\geq8$ then, locally, $(M,E,V,\nabla)$ is the product of 
$({\rm Im}\Hq\!)^{\rank V}$ with a hypercomplex manifold. 
\end{thm} 
\begin{proof} 
By Proposition \ref{prop:f-q_rV>1_rW>0} and Corollary \ref{for:connections_flat_on_Q_and_V} the connections induced 
on $V$ and $Q$ are flat. Furthermore, as in the proof of Proposition \ref{prop:co-cr_q_with_first_Bianchi}  
(note that, here, we do not need \cite[Theorem 2.2]{KMa-BLMS95}\,) we obtain 
that, locally, $V$ is a hypercomplex manifold such that the projection onto $M$ is twistorial.\\ 
\indent  
Moreover, we have that $\p^*\nabla$ restricts to give a flat connection on the quaternionic distribution $K$ on $V$ 
generated by ${\rm ker}\dif\!\p$\,; indeed, we have $K=\p^*\bigl(V\oplus(V\otimes Q)\bigr)$\,. 
Therefore $K$ is integrable and, as $\p^*\nabla$ is, also, torsion free, its leaves are, 
locally, quaternionic vector spaces, whose linear quaternionic structures are preserved by the parallel transport of $\p^*\nabla$. 
Thus, if  $U$ is a covariantly constant section of $K$ and $X$ is a section of $\p^*W$ then $[U,X]=(\p^*\nabla)_UX$ is a section 
of $\p^*W$. Hence, the linear quaternionic structures on the leaves of $K$ are (locally) projectable with respect to $\p^*W$. 
Therefore, locally, there exists a quaternionic submersion $\phi$ from $V$ onto $\Hq^{\!\rank V}$ which, by \cite{IMOP}\,, is twistorial. 
Thus, $\phi$ restricted to $M$, identified with the zero section of $V$, is a twistorial submersion onto $({\rm Im}\Hq\!)^{\rank V}$ whose fibres 
are the leaves of $W$.\\ 
\indent 
Now, as $Q$ is flat, $V$ is locally a hypercomplex manifold and $\p^*\nabla$ is its Obata connection. Let $J$ be any covariantly constant 
admissible complex structure on $V$. Thus, $T^JV={\rm ker}(J+{\rm i})$ is preserved by $\p^*\nabla$. Hence, if $X$ is a section of $T^JV$ 
and $U$ is section of $K$ we have that $[U,X]=(\p^*\nabla)_UX-(\p^*\nabla)_XU$ is a section of $T^J+K$. Therefore $T^JV$ is projectable 
with respect to $K$. This shows that, locally there exists a triholomorphic submersion from $V$ onto a hypercomplex manifold $N$, with $\dim N=\rank W$, 
which factorises into $\p$ followed by a twistorial submersion $\psi$ from $M$ to $N$; also, the latter is triholomorphic when restricted to the leaves of $W$.\\ 
\indent 
Finally, the map $M\to({\rm Im}\Hq\!)^{\rank V}\times N$, $x\mapsto\bigl(\phi(x),\psi(x)\bigr)$ provides the claimed (twistorial) identification.    
\end{proof}

\end{document}